\documentclass[11pt, reqno]{amsart}
\usepackage{amssymb,latexsym}
\usepackage{enumerate}
\usepackage{hyperref}
\usepackage{comment}
\usepackage[margin=1in]{geometry}

\makeatletter \@namedef{subjclassname@2010}{%
  \textup{2010} Mathematics Subject Classification}
\makeatother

\newtheorem{Theorem}{Theorem}[section]

\newtheorem{Lemma}[Theorem]{Lemma}
\newtheorem{Corollary}[Theorem]{Corollary}

\newtheorem*{UnnumberedCorollary}{Corollary}
\newtheorem{Claim}[Theorem]{Claim}

	\newcommand{\RR}[0]{\mathbb R}
\newcommand{\FF}[0]{\mathbb F}

	\newcommand{\ZZ}[0]{\mathbb Z}

	\newcommand{\cP}[0]{\mathcal P}

\newcommand{\cF}[0]{\mathcal F}

\newcommand{\cI}[0]{\mathcal I}	
\newcommand{\cJ}[0]{\mathcal J}

\newcommand{\pp}[0]{\textbf{\textit{p}}}
\newcommand{\qq}[0]{\textbf{\textit{q}}}

\newcommand{\vv}[0]{\textbf{\textit{v}}}
\newcommand{\jj}[0]{\textbf{\textit{j}}}
\newcommand{\llll}[0]{\textbf{\textit{l}}}

\newcommand{\mm}[0]{\textbf{\textit{m}}}\newcommand{\zz}[0]{\textbf{\textit{z}}}

\renewcommand{\bar}[1]{\overline{#1}}

\newcommand{\lr}[1]{\left(#1\right)}

\begin{document}


\baselineskip=17pt



\title{On iterated product sets with shifts}

\author[B. Hanson]{Brandon Hanson} \address{Pennsylvania State University\\
University Park, PA, USA}
\email{bwh5339@psu.edu}
\author[O. Roche-Newton]{Oliver Roche-Newton} \address{Johann Radon Institute for Computational and Applied Mathematics\\
Linz, Austria}
\email{o.rochenewton@gmail.com}

\author[D. Zhelezov]{Dmitrii Zhelezov} \address{Alfr\'{e}d R\'{e}nyi Institute of Mathematics \\ 
Hungarian Academy of Sciences, Budapest, Hungary }
\email{dzhelezov@gmail.com}
\date{}

\begin{abstract} We prove that, for any finite set $A \subset \mathbb Q$ with $|AA| \leq K|A|$ and any positive integer $k$, the $k$-fold product set of the shift $A+1$ satisfies the bound
$$| \{(a_1+1)(a_2+1) \cdots (a_k+1) : a_i \in A \}| \geq \frac{|A|^k}{(8k^4)^{kK}}. $$
This result is essentially optimal when $K$ is of the order $c\log|A|$, for a sufficiently small constant $c=c(k)$. 

Our main tool is a multiplicative variant of the $\Lambda$-constants used in harmonic analysis, applied to Dirichlet polynomials.

\end{abstract}

\maketitle

\section{Introduction}

Let $A$ be a finite set of integers. The \textit{sum set} and \textit{product set} of $A$ are defined respectively as
$$A+A:= \{a+b : a,b \in A\},\,\,\,\,\,\, AA:=\{ab :a,b \in A\}.$$
The sum-product problem is concerned with showing that one of these sets is always large. It was conjectured by Erd\H{o}s and Szemer\'{e}di \cite{ES} that, for all $\epsilon>0$ and any finite $A \subset \mathbb Z$,
\begin{equation}
\max \{|A+A|,|AA|\} \geq c(\epsilon)|A|^{2-\epsilon}
\label{ESconj}
\end{equation}
where $c(\epsilon)>0$ is an absolute constant. The same conjecture can also be made over the reals, and indeed other fields. The  Erd\H{o}s--Szemer\'{e}di conjecture remains open, and it appears to be a deep problem. Konyagin and Shkredov \cite{KS} proved that \eqref{ESconj} holds with $\epsilon <2/3$, and the current best bound, due to Rudnev, Shkredov and Stevens \cite{RSS}, has $\epsilon \leq 2/3 -1/1509 +o(1)$. These bounds hold over real numbers, and their proofs are geometric in nature.

A similar question can also be considered with more variables. The \textit{$k$-fold sumset} and \textit{$k$-fold product set} of $A \subset \mathbb Z$ are defined respectively as
$$kA:= \{a_1+\dots +a_k : a_1,\dots,a_k \in A\},\,\,\,\,\,\, A^{(k)}:=\{a_1\cdots a_k :a_1\dots,a_k \in A\}.$$
Erd\H{o}s and Szemer\'{e}di also made the even more general conjecture that, for all $\epsilon>0$ and any finite $A \subset \mathbb Z$,
\begin{equation}
\max \{|kA|,|A^{(k)}|\} \geq c(\epsilon)|A|^{k-\epsilon}.
\label{ESconjk}
\end{equation}
Given the state of progress with the $k=2$ version of this conjecture, it is not surprising that the more general conjecture is also wide open. However, over the rationals, a series of remarkable results concerning unlimited growth does exist. The first of these results is the following theorem of Chang \cite{chang2003erdHos}.

\begin{Theorem} \label{thm:Chang} Let $A \subset \mathbb Q$ be a finite set with $|AA| \leq K|A|$ and let $k \geq 2$ be an integer. Then
\begin{equation}
|kA| \geq \frac{|A|^k}{(2k^2-k)^{kK}}.
\label{Changbound}
\end{equation}
\end{Theorem}
The results proved in \cite{chang2003erdHos} were actually somewhat more general. It was established that 
\begin{equation} \label{Changenergy2}
E_k^+(A) \leq |A|^k (2k^2-k)^{kK},
\end{equation}
where $E_k^+(A)$ is the \textit{$k$-fold additive energy}, which is defined as the quantity
$$E_k^+(A):=|\{a_1+\dots+a_k=a_{k+1}+\cdots +a_{2k} :a_1,\dots,a_{2k}\in A\}|.$$
Note that the trivial solutions with $a_1=a_{k+1}, \dots, a_k=a_{2k}$ ensure that $E_K^+(A) \geq |A|^k$ and so  \eqref{Changenergy2} is a factor of $(Ck^2)^{kK}$ away from being optimal. Inequality \eqref{Changbound} follows from \eqref{Changenergy2} after an application of the Cauchy--Schwarz inequality. 

Giving a result with even more generality, Chang in fact proved a version of \eqref{Changenergy2} with weighted energy.
See the forthcoming Section \ref{weights} for a discussion on energy and weighted energy featuring the relevant definitions.

In the statement of Theorem \ref{thm:Chang}, we think of $k$ as a fixed constant. The theorem then gives a very strong lower bound for the size of $kA$  when $K$ is significantly smaller than $\log |A|$. However, if we push to the range $K=|A|^{\epsilon}$ for some positive $\epsilon$, then Theorem \ref{thm:Chang} gives only a trivial bound.

In a follow-up paper of Bourgain and Chang \cite{bourgain2004size}, this method was used as a foundation and developed considerably in order to prove the following result. 

\begin{Theorem} \label{thm:BC1} Let $k \geq 2$ be a fixed integer. Given $\gamma >0$, there is a constant $\Lambda= \Lambda(\gamma , k)$ such that for all finite sets $A \subset \mathbb Q$ with $|AA| \leq K|A|$, 
$$|kA| \geq \frac{|A|^{k-\gamma}}{K^{\Lambda}}.$$
\end{Theorem}

In particular, this now gives an excellent lower bound for $|kA|$ when $K=|A|^{\epsilon}$, for some small but positive $\epsilon$.

Once again, a more general version of Theorem \ref{thm:BC1} was actually proved in \cite{bourgain2004size}, giving an upper bound for the weighted energy of $A$. This level of generality was important for establishing the main result of \cite{bourgain2004size}, which was the following result on unbounded growth of sums or products.

\begin{Theorem} \label{thm:BC2}
For all $b \geq 0$, there is an integer $k=k(b)$ such that for all $A \subset \mathbb Q$ with $|A| \geq 2$,
$$\max \{ |kA|, |A^{(k)}|\} \geq |A|^b.$$
\end{Theorem}

See Zhelezov \cite{Z} for an exposition of the work of \cite{chang2003erdHos} and \cite{bourgain2004size}.


In this paper, we consider an alternative formulation of the sum-product problem with products and products of shifts. Given an arbitrary $A \subset \mathbb R$, sum-product heuristics lead one to believe that $\max \{|AA|, |(A+1)(A+1)|\}$ is large. Roughly speaking, this is an assertion that if $A$ is multiplicatively structured then an additive shift will disturb this structure. A similar problem was considered in the finite field setting by Bourgain \cite{bourgain2005}, and it was established by Garaev and Shen \cite{GS} that for any finite $A \subset \mathbb R$,
\begin{equation}
\max \{|AA|, |(A+1)(A+1)|\} \geq c|A|^{5/4},
\label{GS}
\end{equation}
where $c>0$ is an absolute constant. See also Jones--Roche-Newton \cite{JRN} for a slightly improved exponent. In principle, the value of $1$ for the shift is not important, and a shift by any non-zero $x$ should give the same outcome. Indeed, the proof of \eqref{GS} works in exactly the same way if $1$ is replaced with any $x \neq 0$.

It seems plausible that the numerology of the Erd\H{o}s--Szemer\'{e}di conjecture holds for this problem too. That is,  we expect that $\max\{|A^k|, |(A+1)^k|\}$ should be close to $|A|^k$. The main result of this paper proves a result in this direction, although under the stronger assumption that $A$ has small multiplicative doubling. 

\begin{Theorem} \label{thm:main}
Let $A \subset \mathbb Q$ be finite and suppose that $|AA|=K|A|$. Then, for any $k \geq 2$,
$$|(A+1)^{(k)}| \geq \frac{|A|^k}{(8k^4)^{kK}}. $$
\end{Theorem}

This is an analogue of Theorem \ref{thm:Chang} above, and gives an essentially optimal lower bound for $|(A+1)^k|$ in the case when $K < c\log|A|$. Again, we actually prove a more general result with energies and weights. See Theorem \ref{thm:energyrationalcase} for the full statement.

Note that, since this theorem holds for any set of rationals, it follows that the shift value $1$ can be replaced by any rational $\lambda \neq 0$. Indeed, if $|AA|=K|A|$ then also $|(\frac{A}{\lambda})(\frac{A}{\lambda})|=K|A|$ for any rational $\lambda$. We can therefore apply Theorem \ref{thm:main} to the set $A/\lambda $ and conclude that
$$|(A+\lambda)^{(k)}|=\left|\left(\frac{A}{\lambda}+1\right)^{(k)}\right| \geq \frac{|A|^k}{(8k^4)^{kK}}. $$

It seems that an analogue of Theorem \ref{thm:BC2} for products and products of shifts would have several interesting consequences in additive number theory. The authors consider this problem in a forthcoming paper \cite{HRNZ2}.

Finally, it is worth noting that one can prove a weaker version of Theorem~\ref{thm:main} by using quantitative bounds for the Subspace Theorem \cite{evertse2002linear} similarly to \cite{chang2006sum}. The resulting bound will be of the form
$$
|(A+1)^{(k)}| \geq e^{-c(k)K} |A|^k.
$$
Apart from relying on a deep and hard result \cite{evertse2002linear}, the main disadvantage of such an approach is that the dependence $c(k)$ turns out to be triply exponential in $k$ (comparing to almost linear in our case). Indeed, the number of terms after expanding the brackets in $(A+1)^k$ grows exponentially in $k$, while the state-of-the-art quantitative bound for the Subspace Theorem is doubly exponential in the number of variables; see \cite{chang2006sum},\cite{amoroso2009small} for further details. 

On the other hand, the dependence $c(k)$ may be important for applications. We record the following application to $S$-unit equations as an example.

\begin{Corollary} \label{corr:height}
Let $p_1, \ldots, p_r$ be a set of primes and $S$ be a set of rational numbers of the form $$
s = p^{\alpha_1}_1 \ldots p^{\alpha_r}_r
$$
with $|\alpha_i| \leq H$. Then for any rational $c_1, c_2 \neq 0$ the number of solutions $(s_1, s_2) \in S \times S$ to
$$
c_1s_1 + c_2s_2 = 1
$$
is bounded by $(\log H)^{C2^r}$ with some absolute effective constant $C > 0$.
\end{Corollary}

Corollary~\ref{corr:height} is much weaker than the celebrated result of Beukers and Schlikewei~\cite{beukers1996equation}, but may indicate a possible connection between sum-product estimates and Diophantine equations. We prove Corollary~\ref{corr:height} in Section~\ref{sec:corollary}.

\subsection{Outline of the proof}

Our proof strategy is based on the ideas that Chang used to prove Theorem \ref{thm:Chang}. Chang's strategy begins by using Freiman's Lemma to show that the elements of a set $A$ with $|AA|=K|A|$ are determined by the powers of a set $\{p_1,\dots,p_K\}$ of $K$ primes. The aim is to use this information to bound the number of solutions to
\begin{equation}
a_1+\dots+a_k=a_{k+1}+ \dots + a_{2k},
\label{Changenergy}
\end{equation}
such that $a_1,\dots,a_{2k} \in A$. The problem of bounding the number of solutions to \eqref{Changenergy} is converted into the problem of bounding certain trigonometric sums. The proof then proceeds by induction on $K$. If $K=1$ then it follows that some power of the determining prime $p$ must occur more than once in \eqref{Changenergy}. Using this and an application of H\"{o}lder's inequality establishes the base case, and closing the induction is then fairly straightforward.

We imitate this argument, with the role of trigonometric sums played instead by Dirichlet polynomials, since we are now bounding the number of solutions to the multiplicative equation
\begin{equation}
(a_1+1)\cdots(a_k+1)=(a_{k+1}+ 1)\cdots (a_{2k}+1).
\label{Ourenergy}
\end{equation}
When $A$ is a set of integers things go through rather smootly. However, the statement of Theorem \ref{thm:main} is not so satisfying in this case, since the lack of dilation invariance in our problem restricts the value of the additive shift we can take. The case when $A$ consists of rational numbers, and in particular when the powers of our determining prime may be negative, does not yield to this method immediately, and the major challenge we faced was to tackle this case.

To explain how we overcome this hurdle, let us begin with the base case $K=1$. We can control the number of solutions to \eqref{Ourenergy} by those solutions consisting only of positive powers of $p$ and those consisting of only negative powers. We can then bound those solutions coming from the first case by treating them as integers. However, we cannot yet say anything about the solutions coming from negative powers of $p$. We thus get a base case with two seperate terms, and this blows up into a rather complicated statement after the induction process. However, once we apply a little extra information coming from Freiman's Lemma, we are able to bound these terms with only negative powers, and after some tricky calculations the desired result follows.

Note that, for the case $k=2$ of Theorem \ref{thm:Chang}, and indeed also that of Theorem \ref{thm:BC1}, things are somewhat easier, as the whole proof can be carried out in the ``physical space", without the need for trigonometric sums. See \cite{Z} for the details of the modifications to the argument. A similar situation arises here, as the case $k=2$ of Theorem \ref{thm:main} can be proven without using Dirichlet polynomials. However, when $k=2$, a better version of Theorem \ref{thm:main} is available, even over $\mathbb R$; see Shkredov \cite[Theorem 12]{S}.

\subsection{Notation}

For two integers $a,b$ we write, as usual, $(a,b)=1$ if $a$ and $b$ are coprime. Similarly for $a,b \in \mathbb Q$ we say that $a$ and $b$ are coprime if, after writing $a=n_a/d_a$ and $b=n_b/d_b$ as reduced fractions, there is no prime which appears in both $a$ and $b$. That is, $(n_a,n_b)=(n_a,d_b)=(d_a,n_b)=(d_a,d_b)=1$.

We write $\mathbb Z_{ \geq 0}$ for the set of all non-negative integers, and $\mathbb Z_-$ for the set of all negative integers.

\section{Background on Dirichlet Polynomials}
Let $(w_n)_{n=1}^N$ be a finite sequence of non-negative reals. The associated Dirichlet polynomial is
\[\sum_{n}w_nn^{s}=\sum_{n}w_n\exp(s(\log n))\] which is a function of a complex variable $s$.
These polynomials are not periodic in $T$, but still satisfy similar properties to those of trigonometric polynomials when integrated over a long interval.

\begin{Lemma} \label{Dirbasic1}
For any $m,n\in \mathbb Q$ we have 
\[\int_0^T (m/n)^{it}dt=\begin{cases}T&\text{ if }m=n,
\\ O_{m,n}(1)&\text{ if }m\neq n.\end{cases}\]
Consequently, for $k\geq 1$, we have
\[\frac{1}{T}\int_0^T\left|\sum_{n}w_nn^{it}\right|^{2k}dt=\sum_{n_1\cdots n_k=n_{k+1}\cdots n_{2k}}w_{n_1}\cdots w_{n_{2k}}+o_{T\to\infty}(1)\]
where the implied constants in $o_{T\to \infty}(1)$ depend on the sequence $w_n$.
\end{Lemma}

In particular, this leads to the following result which gives a connection between Dirichlet polynomials and the weighted multiplicative energy of $A$.

\begin{Corollary} \label{Dirbasic2}
Let $A$ be a finite set of rationals and let $w=\{w_a : a \in A \}$ be a set of weights on $A$. Then \[E_{k,w}(A)=\lim_{T\to\infty}\frac{1}{T}\int_0^T\left|\sum_{a\in A}w_a a^{it}\right|^{2k}dt.\]

\end{Corollary}

We will be interested in the following sorts of Dirichlet polynomials, and only interested at purely imaginary values. For $p$ a prime and a rational number $x$, let $v_p(x)$ denote the $p$-adic valuation of $x$, that is, the power of $p$ when $x$ is expressed in its reduced form as a product of primes:
$$x=\prod_{p} p^{v_p(x)}.$$  
For any $u \in \mathbb Q$, write $\cF_{p,j,u}$ for the set of those Dirichlet polynomials of the form 
\[f_j(t)=\sum_{n \in \mathbb Q:v_p(n)=j}w_n(n+u)^{it}.\]

\section{Weighted multiplicative energy} \label{weights}
A key notion in this paper is that of the \emph{multiplicative energy of order $k$} of a set $A$, which is defined as 
\[
E_k(A) = \sum_{n} \Gamma_{k, A}^2(n),
\]
where $\Gamma_{k,A}(n)=|\{(a_1,\dots,a_k) \in A^k : a_1\cdots a_k = n \}$. Alternatively, this is the number of solutions to the equation
\[a_1\cdots a_k=a_{k+1} \cdots a_{2k}. \]
Good upper bounds for $E_k(A)$ lead to good lower bounds for $|A^k|$, because of the following inequality which is a standard application of the Cauchy-Schwarz inequality;
\[E_k(A) |A^{(k)}| \geq |A|^{2k} .\]

It turns out that a certain weighted version of $E_k(A)$ is more robust for applications. Let $A$ be a finite set of size $N$ with elements $a_1, \ldots, a_N$ (the actual ordering is not important). Next, let $w= \{w_1, \ldots, w_{N} \}$ be a set of non-negative real numbers. One can think of $w_i$ as the weight attached to the element $a_i$. Then define the multiplicative energy of order $k$ with weights $w_i$ as 
\[
E_{k, w}(A) := \sum w_{i_1} \ldots w_{i_{2k}},
\]
where the summation is taken over all $2k$-tuples $(i_1, \ldots, i_{2k})$ such that
\[
a_{i_1}\ldots a_{i_k} = a_{i_{k+1}}\ldots a_{i_{2k}}.
\]
Finally, define $\Lambda_k(A)$ as
\[
\Lambda_k(A) := \max_{w} E_{k, w}(A)^{1/k},
\]
where the maximum is taken over all weights $w$ such that 
\begin{equation} \label{eq:weightsnorm}
\sum^{N}_{i=1} w_i^2 = 1.
\end{equation}
It is well-defined by compactness. 
An attentive reader will notice that our definition of $\Lambda_k$ is a straightforward `multiplicative' adaptation of the $\Lambda$-constants widely used in harmonic analysis. 

The use of $\Lambda$-constants for sum-product type problems was pioneered by Chang and Bourgain (see \cite{chang2003erdHos}, \cite{bourgain2004size} and references therein). 
The present note is largely based on the ideas of \cite{chang2003erdHos}. In particular, Section~\ref{sec:integer} is an adaptation of the technique used in \cite{chang2003erdHos} to the multiplicative setting.

Now we record some properties of $\Lambda_k(A)$ in order to justify such a quantity. Let $\|\cdot \|_{2k}$ be the standard norm in $L_{2k}[0, T]$, normalised such that $\| 1\|_{2k} = 1$. So,
$$\| f \|_{2k} := \left( \frac{1}{T} \int_0^T|f(t)|^{2k} dt \right)^{1/2k}. $$

\begin{Lemma} \label{lm:anyweights}
 Let $A$ be a finite set with some non-negative real numbers $w_a$ assigned to each element $a \in A$. Then
 \begin{equation} \label{eq:lambdaanyweights}
 \left \| \sum_{a \in A} w_a a^{it} \right\|^2_{2k} \leq \Lambda_k(A) \left(\sum_{a \in A} w_a^2 \right) + o_{T\to\infty}(1).
 \end{equation}
\end{Lemma}
\begin{proof}
If $\sum_{a \in A} w^2_a = 0$ the claim of the lemma is trivial. Otherwise, define new weights
\[
w'_a := \frac{w_a}{(\sum_{a \in A} w_a^2)^{1/2}}
\]
which satisfy (\ref{eq:weightsnorm}). It thus suffices to show that
\[
\left \| \sum_{a \in A} w'_a a^{it} \right\|^2_{2k} \leq \Lambda_k(A) + o_{T\to\infty}(1),
\]
which is a straightforward consequence of our definition of $\Lambda_k(A)$ and Lemma \ref{Dirbasic1}.
\end{proof}

A crucial corollary is the following stability property of $\Lambda_k$ which does not seem to be available when one works with multiplicative energies.
\begin{Corollary} \label{corr:stability}
Let $A' \subset A$. Then
\[
E^{1/k}_k(A') \leq \Lambda_k(A)|A'|.
\]
In particular,
\[
E_k(A) \leq \Lambda^k_k(A)|A|^k.
\]
\end{Corollary}
\begin{proof}
 Apply Lemma \ref{lm:anyweights} with $w_a = 1$ if $a \in A'$ and $w_a = 0$ otherwise.
\end{proof}

\section{The integer case} \label{sec:integer}

In this section we restrict ourselves to the integer case and prove the following bound.
\begin{Theorem} \label{thm:maininteger}
Let $A$ be a sufficiently large but finite set of positive integers, with the property that $|AA|\leq K|A|$ for some integer $K$. Then, for any $u$ coprime with the elements of $A$, we have
\begin{equation}
\Lambda_k(A+u) \leq (2k^2)^{K}.
\label{Ebound}
\end{equation}
In particular,
\begin{equation}
 |(A+u)^{(k)}| \geq \frac{|A|^k}{(2k^2)^{kK}}.
\label{Cbound}
\end{equation}

\end{Theorem}

The following is analogous to Proposition 6 of Chang \cite{chang2003erdHos}. Here, we have an extra restriction that we are dealing only with positive powers of a given prime $p$. This makes things rather simpler for us, and we can follow the strategy of Chang, making the appropriate minor alterations.
\begin{Lemma} \label{prop:basecaseint}
Let $p$ be a prime and let $\cJ$ be a set of positive integers. Let $u \in \mathbb Z$ such that $(p,u)=1$. Suppose $f_j\in \cF_{p,j,u}$ for $j\in \cJ$. Then for $k\geq 1$,
\begin{equation}
\lim_{T\to \infty}\lr{\frac{1}{T}\int_0^T\left|\sum_j f_j(t)\right|^{2k}dt}^{1/k}\leq 2k^2\sum_j \lim_{T\to \infty}\lr{\frac{1}{T}\int_0^T\left|f_j(t)\right|^{2k}dt}^{1/k}.
\label{basecase}
\end{equation}
\end{Lemma}
\begin{proof}
Expanding the $k$'th power of the left hand side of \eqref{basecase} gives
\begin{equation}\sum_{j_1,\dots,j_{2k}} \lim_{T\to \infty}  \frac{1}{T}\int_0^T f_{j_1}(t) \cdots f_{j_k}(t) \bar{f_{j_{k+1}}(t)} \cdots \bar{f_{j_{2k}}(t)} dt
\label{basiceq1}
\end{equation}
For fixed $j_1, \dots, j_{2k}$, the quantity
\[\lim_{T\to \infty} \frac{1}{T} \int_0^T f_{j_1}(t) \cdots f_{j_k}(t) \bar{f_{j_{k+1}}(t)} \cdots \bar{f_{j_{2k}}(t)} dt \]
gives a weighted count of the number of solutions to equations of the form
\begin{equation}
(m_1p^{j_1}+u)\cdots(m_kp^{j_k}+u)=(m_{k+1}p^{j_{k+1}}+u)\cdots(m_{2k}p^{j_{2k}}+u),
\label{basiceq}
\end{equation}
with each of the $m_i \in \mathbb Q$ coprime to $p$. We claim that there are no solutions to \eqref{basiceq} if all of the $j_i$ are distinct. Indeed, suppose that all of these powers are distinct, and in particular there is a unique smallest power, say $j_1 < j_2,j_3,\dots,j_{2k}$. Then we get (expanding the brackets, subtracting $u^k$ from both sides, and multiplying out all denominators of the $m_i$) an equation 
\begin{equation}
Mp^{j_1}+N_1p^{j_1+1} = N_2p^{j_1+1}
\end{equation}
for some integers $M,N_1$ and $N_2$, such that $(M,p)=1$. This is a contradiction, since the right hand side is divisible by $p^{j_1+1}$ and the left hand side is not. Therefore, there are no contributions to \eqref{basiceq1} coming from those terms where all of the $j_i$ are distinct.

It remains to consider the cases where two or more of the powers $j_1, \dots, j_{2k}$ are the same. There are three kinds of ways in which this can happen.

\begin{enumerate}
\item $j_i=j_i'$ with $1 \leq i \leq k$ and $k+1 \leq i' \leq 2k$. There are $k^2$ possible positions for such a pair $(i,i')$,
\item $j_i=j_i'$ with $1 \leq i, i' \leq k$. There are $\binom{k}{2}$ possible positions for such a pair $(i,i')$,
\item $j_i=j_i'$ with $k+1 \leq i, i' \leq 2k$. There are $\binom{k}{2}$ possible positions for such a pair $(i,i')$.
\end{enumerate}

Suppose we are in situation (1) above. Specifically, suppose that $j_1=j_{2k}$. The other $k^2-1$ cases can be dealt with by the same argument. Then these terms in \eqref{basiceq1} can be rewritten as
\begin{align*} & \sum_{j_1} \lim_{T\to \infty} \frac{1}{T} \int_0^T  f_{j_1}(t) \bar{f_{j_{1}}(t)} \sum_{j_2,\dots, j_{2k-1}} f_{j_2}(t) \cdots f_{j_k}(t)  \bar{f_{j_{k+1}}(t)} \cdots \bar{f_{j_{2k-1}}(t)} dt 
\\& = \sum_j \lim_{T\to \infty}\frac{1}{T}\int_0^T  |f_j(t)|^2\left|\sum_{j} f_j(t)\right|^{2(k-1)} dt.
\end{align*}

Suppose we are in situation (2). Specifically, suppose that $j_1=j_{2}$. The other $ \binom{k}{2}$ cases can be dealt with by the same argument. Then these terms in \eqref{basiceq1} can be rewritten as

\begin{align*} & \sum_{j_1} \lim_{T\to \infty} \frac{1}{T} \int_0^T  f_{j_1}^2(t)  \sum_{j_3,\dots, j_{2k}} f_{j_3}(t) \cdots f_{j_k}(t)  \bar{f_{j_{k+1}}(t)} \cdots \bar{f_{j_{2k}}(t)} dt 
\\& \leq  \sum_j \lim_{T\to \infty}\frac{1}{T}\int_0^T  |f_j(t)|^2\left|\sum_{j} f_j(t)\right|^{k-2} \left|\sum_{j} \bar{f_j(t)}\right|^{k}  dt
\\&= \sum_j \lim_{T\to \infty}\frac{1}{T}\int_0^T  |f_j(t)|^2\left|\sum_{j} f_j(t)\right|^{2(k-1)}   dt.
\end{align*}

The same argument also works in case (3). Returning to \eqref{basiceq1}, we then have
 \begin{align*} & \lim_{T\to \infty}\lr{\frac{1}{T}\int_0^T\left|\sum_j f_j(t)\right|^{2k}dt}
 \\&\sum_{j_1,\dots,j_{2k}} \lim_{T\to \infty}  \frac{1}{T}\int_0^T f_{j_1}(t) \cdots f_{j_k}(t) \bar{f_{j_{k+1}}(t)} \cdots \bar{f_{j_{2k}}(t)} dt
 \\& \leq \left(k^2 + 2\binom{k}{2}\right) \sum_j \lim_{T\to \infty}\frac{1}{T}\int_0^T  |f_j(t)|^2\left|\sum_{j} f_j(t)\right|^{2(k-1)}   dt
  \\& \leq 2k^2 \sum_j \lim_{T\to \infty}\frac{1}{T}\int_0^T  |f_j(t)|^2\left|\sum_{j} f_j(t)\right|^{2(k-1)}   dt.
\end{align*}
Here we have used the fact that the weights are positive and real, which allows us to overcount solutions which may occur in more than one of the three cases above. Finally, an application of H\"{o}lder's inequality gives
\[ \lim_{T\to \infty}\lr{\frac{1}{T}\int_0^T\left|\sum_j f_j(t)\right|^{2k}dt} \leq 2k^2 \sum_j  \lim_{T\to \infty} \lr{\frac{1}{T}\int_0^T|f_j(t)|^{2k}dt}^{1/k}\lr{\frac{1}{T}\int_0^T\left|\sum_{j}f_j(t)\right|^{2k}dt}^{1-1/k}. \]
A rearrangement of this inequality completes the proof.

\end{proof}

Using an inductive argument, we can extend this result to handle generalized geometric progressions. To that end, for a vector $\pp=(p_1,\ldots,p_t)$ of prime numbers and a vector $\jj=(j_1,\ldots,j_t)$ of non-negative integers, let $\cF_{\pp,\jj, u}$ be the set of Dirichlet polynomials of the form
\[
f_\jj(t)=\sum_{n:v_{p_i}(n)=j_i}w_n(n+u)^{it}.
\]

\begin{Lemma} \label{mainlemint}
Let $\pp=(p_1,\ldots,p_K)$ be a vector of prime numbers and let $\cJ$ be a set of vectors with non-negative entries. Let $u \in  \mathbb Z$ such that $(u,p_1\cdots p_K)=1$. Suppose $f_\jj\in\cF_{\pp,\jj, u}$ for $\jj\in \cJ$. Then for $k\geq 1$, we have
\begin{equation} \label{eq:mainint}
\lim_{T\to \infty}\lr{\frac{1}{T}\int_0^T\left|\sum_{\jj\in\cJ} f_\jj(t)\right|^{2k}dt}^{1/k}\leq (2k^2)^K\sum_{\jj\in \cJ}\lim_{T\to \infty}\lr{\frac{1}{T}\int_0^T\left|f_\jj(t)\right|^{2k}dt}^{1/k}.
\end{equation}
\end{Lemma}
\begin{proof}
We establish the convention that $f_{\jj}$ is identically zero for all $\jj \notin \cJ$. Therefore, we can complete the sum in \eqref{eq:mainint}, and the aim is to prove that
\[
\lim_{T\to \infty}\lr{\frac{1}{T}\int_0^T\left|\sum_{\jj\in \mathbb Z_{\geq 0}^K} f_\jj(t)\right|^{2k}dt}^{1/k}\leq (2k^2)^K\sum_{\jj\in \mathbb Z_{\geq 0}^K}\lim_{T\to \infty}\lr{\frac{1}{T}\int_0^T\left|f_\jj(t)\right|^{2k}dt}^{1/k}.
\]

We proceed by induction on $K$, the base case $K=1$ being given by Lemma \ref{prop:basecaseint}. Then
\begin{align*}
&\lim_{T\to \infty} \left(\frac{1}{T} \int_0^T \left|\sum_{\jj \in \mathbb Z_{\geq 0}^K }f_\jj(t) \right|^{2k} dt\right)^{1/k}
\\&=\lim_{T\to \infty} \left(\frac{1}{T} \int_0^T \left|\sum_{j_K\in \mathbb Z_{\geq 0}}\left(\sum_{\jj'\in \mathbb Z_{\geq 0}^{K-1} }f_{\jj',j_K}(t) \right) \right|^{2k} dt\right)^{1/k}
\\&\leq 2k^2\sum_{j_K \in \mathbb Z_{\geq 0}} \lim_{T\to \infty} \left(\frac{1}{T} \int_0^T \left|\sum_{\jj'\in \mathbb Z_{\geq 0}^{K-1} }f_{\jj',j_K}(t) \right|^{2k} dt\right)^{1/k}
\\& \leq 2k^2\sum_{j_K \in \mathbb Z_{\geq 0}} (2k^2)^{K-1} \sum_{\jj'\in \mathbb Z_{\geq 0}^{K-1} }\lim_{T\to \infty} \left(\frac{1}{T} \int_0^T \left|f_{\jj',j_K}(t) \right|^{2k} dt\right)^{1/k}
\\&=(2k^2)^K \sum_{\jj\in \mathbb Z_{\geq 0}^K}\lim_{T\to \infty}\lr{\frac{1}{T}\int_0^T\left|f_\jj(t)\right|^{2k}dt}^{1/k}.
\end{align*}
The first inequality above follows from an application of Lemma \ref{prop:basecaseint}, using the fact that 
\[
\sum_{\jj'\in \mathbb Z_{\geq 0}^{K-1} }f_{\jj',j_K}(t)  \in \cF_{p_K,j_K,u}.
\]
The second inequality follows from the induction hypothesis.
\end{proof}

We are now in a strong position to establish results on the physical side, starting with the following lemma.

\begin{Lemma} \label{physgen} 
Let $\pp=(p_1,\ldots,p_K)$ be a vector of prime numbers and let $\cJ \subset \mathbb Z^K$ be a set of vectors with non-negative entries. Let $u \in  \mathbb Z$ such that $(u,p_1\dots p_K)=1$. Suppose that 
$$A=\bigcup_{\jj=(j_1,\dots,j_K) \in \cJ} \{p_1^{j_1}\dots p_K^{j_K} x(\jj) \}$$
with each $x(\jj) \in \mathbb Q$ coprime to $p_1\cdots p_K$. Then
$$
\Lambda_k(A+u) \leq (2k^2)^K.
$$
\end{Lemma}

\begin{proof}
For each $\jj=(j_1, \dots, j_K) \in \cJ$, define $a_{\jj}=p_1^{j_1}\dots p_K^{j_K} x(\jj)$, so that $A= \{a_{\jj} : \jj \in \cJ \}$. Define $f_\jj(t)=w_\jj(a_\jj+u)^{it}$ for some weights $w=\{w_{\jj} : \jj \in \cJ\}$ satisfying 
$$
\sum_{\jj \in \cJ} w_\jj^2 = 1. 
$$ 
Note that $f_\jj(t)\in\cF_{\pp,\jj,u}$, because of the divisibility conditions in the statement of the lemma. Note also that  
\[
\lim_{T\to \infty} \frac{1}{T}\int_0^T|f_{\jj}(t)|^{2k}dt= w_\jj^{2k}.
\] 
Using this, as well as Theorem \ref{mainlemint} and Corollary~\ref{Dirbasic2}, we conclude that

\begin{align*}
E_{k, w}(A+u)^{1/k} & = \lim_{T\to \infty}\lr{\frac{1}{T}\int_0^T\left|\sum_{\jj\in\cJ} w_\jj(a_\jj +u)^{it}\right|^{2k} dt}^{1/k}
\\&= \lim_{T\to \infty}\lr{\frac{1}{T}\int_0^T\left|\sum_{\jj\in\cJ} f_\jj(t)\right|^{2k}dt}^{1/k}
\\& \leq (2k^2)^K\sum_{\jj\in \cJ}\lim_{T\to \infty}\lr{\frac{1}{T}\int_0^T\left|f_\jj(t)\right|^{2k}dt}^{1/k}
\\&=(2k^2)^K.
\end{align*}
Since the bound above does not depend on the weights, we are done.

\end{proof}

Before completing the proof of Theorem~\ref{thm:maininteger}, we need the following more or less obvious lemma.
\begin{Lemma} \label{linalg}
Let $L$ be an affine subspace of $\FF^l$ of dimension $d$. Then, after relabeling the coordinates if necessary, $L$ has the form
$$L=\{(x_1,\dots,x_d, f_1(x_1,\dots,x_d), \dots, f_{l-d}(x_1,\dots, x_d) : x_1,\dots,x_d \in \FF \},$$
where each of the functions $f_i$ have degree at most $1$ in variables $x_1,\dots,x_d$.

In particular, there are $d$ coordinate projections $\pi_{1},\ldots,\pi_{d}$ such that the map $\pi:\FF^l\to\FF^d$ given by $\pi(\vv)=(\pi_{1}(\vv),\ldots,\pi_{d}(\vv))$ is injective on $L$.
\end{Lemma}
\begin{proof}  
By shifting if necessary, we may assume $L$ is a linear subspace without loss of generality. Let $v_1, \ldots, v_d$ be $d$ linearly independent vectors in $L$ which we place as rows in a matrix $M$ of dimension $d \times l$ and rank $d$. Performing row reduction on $M$ and permuting columns if necessary, we get a matrix $M'$ in a row echelon form.
Since the rank of $M'$ is $d$, the left-most minor of $M'$ is a $d \times d$ identity matrix. The rows of $M'$ span $L$, and this is exactly the claim of the lemma.   
\end{proof}

Let $A$ be a set of positive integers and let $\cP= \{p_1,\ldots,p_t\}$ be the set of primes dividing any element of $A$. Abusing notation slightly, we define a map $\cP:A\to\ZZ^t$ where $\cP(a)=(v_{p_1}(a),\ldots,v_{p_t}(a))$. Denoting by $\cP(X)$ the image of a set $X$ under $\cP$, observe that $\cP(AA)=\cP(A)+\cP(A)$. We define the multiplicative dimension of $A$ to be the least dimension of an affine space $L$ containing $\cP(A)$.
\begin{Theorem}[Freiman's Lemma (see Lemma 5.13 in \cite{TaoVu})]
Let $A\subset \RR^m$ be a finite set not contained in a proper affine subspace. Then \[|A+A|\geq (m+1)|A|-O_m(1).\]
\end{Theorem}

Theorem \ref{thm:maininteger} now becomes a simple corollary.

\begin{proof}[Proof (of Theorem~\ref{thm:maininteger}).]
It follows from Freiman's Lemma that if $|AA|\leq K|A|$ with $|A|$ sufficiently large, then $A$ has multiplicative dimension at most $K$. Thus, there is an affine subspace of $\RR^t$ containing $v(A)$ and of dimension at most $K$. Permuting the coordinates if necessary, it follows from an application of Lemma \ref{linalg} that each $a\in A$ can be written as $a=p_1^{v_{p_1}(a)}\cdots p_K^{v_{p_K}(a)}n_a$ where $n_a$ is not divisible by any $p_i$ with $1\leq i\leq K$ and the vector $(v_{p_1}(a),\ldots,v_{p_K}(a))$ is unique to $a$. This is precisely the situation of Lemma \ref{physgen}. This proves \eqref{Ebound}.

From the Cauchy-Schwarz inequality and Corollary \ref{corr:stability}
\[
|A|^{2k}\leq |(A+u)^{(k)}|E_k(A+u) \leq (2k^2)^{kK}|A|^k|(A+u)^{(k)}|,
\]
whence
\[|A|^{k}\leq (2k^2)^{kK}|(A+u)^{(k)}|.\]
\end{proof}

\section{The rational case - the main lemma}

In this section, we begin to deal with the rational case, which is rather more complicated. The first aim is prove an analogue of Lemma \ref{mainlemint}. We begin to tackle this by proving the following lemma which helps us to decompose $\cJ$ suitably.

\begin{Lemma} \label{lem:split}
Let $\cJ \subset \mathbb Z^K$ and decompose it as $\cJ=\cJ_1 \cup \cdots \cup\cJ_N$. Let $\pp=(p_1, \dots,p_K)$ be a vector of prime numbers and for each $\jj \in \cJ$ let $f_{\jj} \in \mathcal F_{\jj,\pp,1}$. Then
\begin{equation}\label{split}\lim_{T\to \infty}\lr{\frac{1}{T}\int_0^T\left|\sum_{\jj\in\cJ} f_{\jj}(t)\right|^{2k}dt}^{1/k} \leq N\sum_{n=1}^N \lim_{T\to \infty}\lr{\frac{1}{T}\int_0^T\left| \sum_{\jj\in \cJ_n}f_{\jj}(t)\right|^{2k}dt}^{1/k}.
\end{equation}
\end{Lemma}

\begin{proof} 
It suffices to prove the inequality for all sufficiently large $T$, which we assume fixed for now.
Then
\begin{equation} \label{eq:Lknormsum}
\lr{\frac{1}{T}\int_0^T\left|\sum_{\jj\in\cJ} f_{\jj}(t)\right|^{2k}dt}^{1/k} =
 \left(\left\| \sum_{n=1}^N\sum_{\jj \in \cJ_n}  f_{\jj} \right\|_{2k} \right)^2\leq\left(\sum_{n=1}^N\left\| \sum_{\jj \in \cJ_n}  f_{\jj} \right\|_{2k} \right)^2,
\end{equation}
by the triangle inequality. By the Cauchy-Schwarz inequality, (\ref{eq:Lknormsum}) is bounded by 
\begin{equation}
N\sum_{n=1}^N\left\| \sum_{j \in \cJ_n}  f_j \right\|_{2k}^2.
\end{equation}
Letting $T \to \infty$ we get the claim of the lemma.
\end{proof}

We need to introduce some notation which will be convenient for the forthcoming statement and its proof. Let $S \subset \mathbb [K]$ for a fixed positive integer $K$. For a set $\cJ\subset\ZZ^K$, we write $\cJ_S\subset \cJ$ for the set of all vectors in $\cJ$ whose non-negative entries lie exclusively in the positions corresponding to elements of $S$. We will let $\pi_S$ denote the projection onto the coordinates from $S$. Suppose $\jj_S\in \pi_S(\cJ)$, then we define 
\[\cJ_S(\jj_S)=\{\jj\in \cJ:\pi_S(\jj)=\jj_S, \pi_{[K] \setminus S}(\jj) \in \mathbb Z_-^{K-|S|} \}.\]
\begin{Theorem} \label{lemmamainrat}
Let $\pp=(p_1,\ldots,p_K)$ be a vector of prime numbers and let $\cJ \subset \mathbb Z^K$ be a finite set of vectors. Suppose $f_\jj\in\cF_{\pp,\jj, 1}$ for $\jj\in \cJ$. Then for $k\geq 1$, we have
\[\lim_{T\to \infty}\lr{\frac{1}{T}\int_0^T\left|\sum_{\jj\in\cJ} f_\jj(t)\right|^{2k} dt}^{1/k}\leq (4k^2)^K \sum_{S \subset \{1,\dots,K\}} \sum_{\jj_S \in \pi_S(\cJ_S)}\lim_{T\to \infty}\lr{\frac{1}{T}\int_0^T\left|\sum_{\jj \in \cJ_S(\jj_S)}f_{\jj}(t)\right|^{2k} dt}^{1/k}.\]
\end{Theorem}
\begin{proof}
First we partition \[\cJ=\bigcup_{S\subseteq \{1,\ldots,K\}}\cJ_S.\]
By Lemma \ref{lem:split}, we have 
\[\lim_{T\to \infty}\lr{\frac{1}{T}\int_0^T\left|\sum_{\jj\in\cJ} f_{\jj}(t)\right|^{2k}dt}^{1/k} \leq 2^K\sum_{S\subseteq\{1,\ldots,K\}} \lim_{T\to \infty}\lr{\frac{1}{T}\int_0^T\left| \sum_{\jj\in \cJ_S}f_{\jj}(t)\right|^{2k}dt}^{1/k}.\]

Now consider the set $\pi_S(\cJ_S)$ which consists of vectors in $\ZZ^{|S|}$, each with non-negative entries. Thus, applying Lemma \ref{mainlemint} to these, we get
\[\lim_{T\to \infty}\lr{\frac{1}{T}\int_0^T\left| \sum_{\jj\in \cJ_S}f_{\jj}(t)\right|^{2k}dt}^{1/k}\leq (2k^2)^{|S|} \sum_{\jj_S \in \pi_S(\cJ_S)}\lim_{T\to \infty}\lr{\frac{1}{T}\int_0^T\left|\sum_{\jj \in \cJ_S(\jj_S)}f_{\jj}(t)\right|^{2k} dt}^{1/k}\]
and the theorem follows.

\end{proof}

\section{The rational case - concluding the proof}

This section is devoted to proving the main result of this paper, which we now state in full.

\begin{Theorem} \label{thm:energyrationalcase}
Let $A \subset \mathbb Q$ be a finite set, with the property that $|AA|\leq K|A|$ for some integer $K$. Then we have
\begin{equation}
\Lambda_k(A+1) \leq (8k^4)^{K}.
\label{EboundQ}
\end{equation}
In particular,
\begin{equation}
 |(A+1)^{(k)}| \geq \frac{|A|^k}{(8k^4)^{kK}}.
\label{CboundQ}
\end{equation}
\end{Theorem}

\begin{proof}

Suppose that $A \subset \mathbb Q$ and $|AA| \leq K|A|$ and let $\mathcal P=\{p_1,\dots,p_l\}$ be the set of all primes that appear in the prime decomposition of the elements of $A$. So, each element $a \in A$ can be expressed uniquely as $a=p_1^{v_{p_1}(a)}\cdots p_l^{v_{p_l}(a)}$. We have no control over the size of $\mathcal P$, but we do know that $\mathcal P$ is finite.
 
As in the proof of Theorem \ref{thm:maininteger}, we also let $\mathcal P: A \rightarrow \mathbb Z^n$ denote the prime evaluation map defined by $\mathcal P(p_1^{v_{p_1}}(a)\cdots p_n^{v_{p_t}}(a))=(v_{p_1}(a),\dots,v_{p_t}(a))$ and note that $|\mathcal P(A) + \mathcal P(A)|=|AA|$.
 
By Freiman's Lemma $\mathcal P(A)$ has affine dimension at most $K$. Then, by Lemma \ref{linalg}, we have that
$$\mathcal P(A)=\{(j_1,\dots,j_K, f_{K+1}(j_1,\dots,j_K), \dots, f_{l}(j_1,\dots, j_K) : (j_1,\dots,j_K) \in \pi(\mathcal P(A)) \},$$
where $\pi$ in the projection from $\mathbb R^l$ to the first $K$ coordinates, and each $f_i$ is either constant, or linear in variables $j_1,\dots,j_K$.

We then apply the inverse map $\mathcal P^{-1}$ and see that $A$ has the form
$$A=\{p_1^{j_1}\dots p_K^{j_K} x(j_1,\dots,j_K) : (j_1,\dots,j_K) \in \pi(\mathcal P(A)) \} ,$$
where $x(j_1,\dots,j_K)=p_{K+1}^{f_{K+1}(j_1,\dots,j_K)}\cdots p_{t}^{f_{l}(j_1,\dots,j_K)}$. Note that $x$ is coprime to $p_1\cdots p_K$.
Let $\cJ =\pi(\mathcal P(A))$ and for $\jj=(j_1,\dots,j_K) \in \cJ$, let 
$$a_{\jj}=p_1^{j_1}\dots p_K^{j_K} x(\jj).$$

Now suppose $w$ is any weight function defined on $A+1$ and we consider
\[\sum_{a\in A}w_{a+1}(a+1)^{it}= \sum _{\jj \in \cJ} w_{a_{\jj}+1}(a_{\jj} +1)^{it}=\sum_{S\subset\{1,\ldots,K\}}\sum_{\jj_S\in\pi_S(\cJ_S)}\sum_{\jj\in \cJ_S(\jj_S)}w_{a_\jj+1}(a_{\jj}+1)^{it}.\]
By Theorem \ref{lemmamainrat},
\begin{multline}\lim_{T\to\infty}\left(\frac{1}{T}\int_0^T\left|\sum_{a\in A}w_{a+1}(a+1)^{it}\right|^{2k}dt\right)^{1/k}\\
\leq(4k^2)^K\sum_{S\subset\{1,\ldots,K\}}\sum_{\jj_S\in\pi_S(\cJ_S)}\lim_{T\to\infty}\left(\frac{1}{T}\int_0^T\left|\sum_{\jj\in \cJ_S(\jj_S)}w_{a_\jj+1}(a_{\jj}+1)^{it}\right|^{2k}dt\right)^{1/k}. \label{close}
\end{multline}
If we define \[A_{S,\jj_S}=\{a_\jj:\jj\in \cJ_S(\jj_S)\}\] and $w_{S,\jj_S}$ to be the weights $w$ restricted to $A_{S,\jj_S}+1$ then the innermost quantity above is
\[\lim_{T\to\infty}\left(\frac{1}{T}\int_0^T\left|\sum_{\jj\in \cJ_S(\jj_S)}w_{a_\jj+1}(a_{\jj}+1)^{it}\right|^{2k}dt\right)^{1/k}=E_{k,w_{S,\jj_S}}(A_{S,\jj_S}+1)^{1/k}.\]
This energy is a weighted count of solutions to the equation
\begin{equation}\label{collision}(a_{\jj_1}+1)\cdots(a_{\jj_k}+1)=(a_{\jj_{k+1}}+1)\cdots(a_{\jj_{2k}}+1).\end{equation}
Define $ \llll_S$ to be the $K$-dimensional vector obtained from $\jj_S$ by adding zero entries to those positions not corresponding to $S$. We can then write $\jj\in\cJ_S(\jj_S)$ as $\llll_S + \jj'$ where $\jj'$ has negative entries on coordinates from $S'=\{1,\ldots,K\}\setminus S$ and zeros on coordinates in $S$. We let $\pp_S$ denote the vector of primes with indices from $S$ and $\qq$ denote the vector of primes with indices from $S'$. Finally, if we use the notation $\zz^\mm=z_1^{m_1}\cdots z_K^{m_K}$, then for $\jj_i=\llll_S + \jj_i'$ we have
\[a_{\jj_i}=\pp_S^{\llll_S}\qq^{\jj_i'}x(\jj_i).\]
Now we claim the following:
\begin{Claim}
If (\ref{collision}) holds then
\begin{equation}
\jj_1'+\cdots+\jj_k'=\jj_{k+1}'+\cdots+\jj_{2k}'
\label{claim1}
\end{equation}
and
\begin{equation}
x(\jj_1)\cdots x(\jj_k)=x(\jj_{k+1})\cdots x(\jj_{2k}).
\label{claim2}
\end{equation}
In particular
\[a_{\jj_1}\cdots a_{\jj_k}=a_{\jj_{k+1}}\cdots a_{\jj_{2k}}.\]
\end{Claim}
\begin{proof}[Proof of claim]
We begin with \eqref{claim1}. We have
\[a_{\jj_i}+1=\frac{\pp_S^{\llll_S}x(\jj_i)+\qq^{-\jj_i'}}{\qq^{-\jj_i'}}.\] The non-zero entries of $-\jj_i'$ are positive and correspond to indices in $S'$. Furthermore $\pp_S^{\llll_S}x(\jj_i)$ is coprime to $p_t$ for any $t\in S'$. Thus \[v_{p_t}(a_{\jj_i}+1)=v_{p_t}(\qq^{\jj_i'}),\] and from this \eqref{claim1} follows. 

Next, we recall that $x(\jj_i)$ is of the form $p_{K+1}^{f_{K+1}(\jj_i)}\cdots p_l^{f_l(\jj_i)}$ for primes $p_{K+1},\ldots,p_l$ and functions $f_i$ which are linear or constant. Thus 
\[v_{p_t}(x(\jj_1)\cdots x(\jj_k))=f_t(\jj_1)+\cdots+f_t(\jj_k)\]
and
\[v_{p_t}(x(\jj_{k+1})\cdots x(\jj_{2k}))=f_t(\jj_{k+1})+\cdots+f_t(\jj_{2k})\]
for any $t\in\{K+1,\ldots,l\}$.
Writing $f_t=H_t+c_t$ for a linear form $H_t$ and a constant $c_t$ we see
\[v_{p_t}(x(\jj_1)\cdots x(\jj_k))=H_t(\jj_1+\cdots+\jj_k)+kc_t\]
and
\[v_{p_t}(x(\jj_{k+1})\cdots x(\jj_{2k}))=H_t(\jj_{k+1}+\cdots+\jj_{2k})+kc_t.\]
In view of the first part of the claim and the the fact that $\jj_i=\llll_S\oplus \jj_i'$, we have \[\jj_1+\cdots+\jj_k=\jj_{k+1}+\cdots+\jj_{2k}\]
and \eqref{claim2} follows.
\end{proof}

Having established this claim, we can multiply both sides of (\ref{collision}) by \[\prod_{i=1}^{k}a_{\jj_i}^{-1}=\prod_{i=k+1}^{2k}a_{\jj_i}^{-1}\]
to get the equation
\[(1+a_{\jj_1}^{-1})\cdots(1+a_{\jj_k}^{-1})=(1+a_{\jj_{k+1}}^{-1})\cdots(1+a_{\jj_{2k}}^{-1}).\]
Counting such equations with the same weights, we are now evaluating
\[E_{k,w_{S,\jj_S}}(A_{S,\jj_S}^{-1}+1)^{1/k}=\lim_{T\to\infty}\left(\frac{1}{T}\int_0^T\left|\sum_{\jj\in \cJ_S(\jj_S)}w_{a_\jj+1}(a_{\jj}^{-1}+1)^{it}\right|^{2k}dt\right)^{1/k}.\]
Crucially, the powers of primes indexed by $S'$ are now all positive in the above quantity, and we can apply Lemma \ref{mainlemint}. The above is thus at most
\[ (2k^2)^{K-|S|} \sum_{\jj\in \cJ_S(\jj_S)}\lim_{T\to\infty}\left(\frac{1}{T}\int_0^T\left|w_{a_\jj+1}(a_{\jj}^{-1}+1)^{it}\right|^{2k}dt\right)^{1/k}\leq (2k^2)^{K}\sum_{\jj\in \cJ_S(\jj_S)}|w_{a_\jj+1}|^2.\]
Inserting this into \eqref{close}, we conclude that
\begin{align*}\lim_{T\to\infty}\left(\frac{1}{T}\int_0^T\left|\sum_{a\in A}w_{a+1}(a+1)^{it}\right|^{2k}dt\right)^{1/k} &\leq(4k^2)^K\sum_{S\subset\{1,\ldots,K\}}\sum_{\jj_S\in\pi_S(\cJ_S)}(2k^2)^{K}\sum_{\jj\in \cJ_S(\jj_S)}|w_{a_\jj+1}|^2. 
\\&= (8k^4)^K\sum_{\jj\in \cJ}|w_{a_{\jj}+1}|^2.
\end{align*}
Since the inequality above is true for any set of weights on $A+1$, \eqref{EboundQ} then follows from the definition of $\Lambda_k(A+1)$.

To prove \eqref{CboundQ}, we apply the Cauchy-Schwarz inequality and Corollary \ref{corr:stability}, just as we did in the conclusion of the proof of Theorem \ref{thm:maininteger}. This yields
$$|A|^{2k} \leq E_k(A+1) |(A+1)^{(k)}| \leq \Lambda_k^k(A+1)|A|^k |(A+1)^{(k)}| \leq (8k^4)^kK |A|^k |(A+1)^{(k)}|, $$
and a rearrangement of this inequality completes the proof of \eqref{CboundQ}.

\end{proof}

\section{Proof of Corollary ~\ref{corr:height}} \label{sec:corollary}

Recall the statement of Corollary~\ref{corr:height}. 
\begin{UnnumberedCorollary}
Let $p_1, \ldots, p_r$ be a set of primes and $S$ be a set of rational numbers of the form $$
s = p^{\alpha_1}_1 \ldots p^{\alpha_r}_r
$$
with $|\alpha_i| \leq H$. Then for any rational $c_1, c_2 \neq 0$ the number of solutions $(s_1, s_2) \in S \times S$ to
\begin{equation} \label{eq:uniteq}
c_1s_1 + c_2s_2 = 1
\end{equation}
is bounded by $(\log H)^{C2^r}$ with some absolute effective constant $C > 0$.   
\end{UnnumberedCorollary}
\begin{proof}
Let $S_1$ be the set of $s_1 \in S$ such that $(s_1, s_2)$ is a solution to (\ref{eq:uniteq}) for some $s_2 \in S$. By the  hypothesis $S_1$ is contained in the generalised geometric progression 
$$
G := \{ p^{i_1}_1 \ldots p^{i_r}_r:\,i_1, \ldots i_r \in [-H, H] \}.
$$
It is straightforward to check that 
$$
|GG| \leq 2^r|G|,
$$
and thus by Corollary~\ref{corr:stability} and Theorem~\ref{thm:energyrationalcase} applied for the set $c_1S_1$
\begin{equation} \label{eq:s1energybound}
E_k(c_1S_1-1) \leq \exp(Ck \log k 2^r)|S_1|^k.
\end{equation}
On the other hand, the shifted set $c_1S_1 - 1$ is contained in $-c_2G$, so
\begin{equation}  \label{eq:s1sizebound}
|(c_1S_1-1)^{(k)}| \leq |G^{(k)}| < \exp(r\log(3kH)).
\end{equation}
Applying Cauchy-Schwarz to (\ref{eq:s1energybound}) and combining with (\ref{eq:s1sizebound}) we have
$$
|S_1| < \exp(C \log k 2^r + r\log(3kH)/k).
$$
Taking 
$$
k = \frac{\log H}{ \log \log H}
$$
and assuming wlog $H \gg r \gg 1$, we get 
$$
|S_1| \leq \exp(C'2^{r} \log \log H) = (\log H)^{C'2^{r}} 
$$
for some absolute $C' > 0$.
\end{proof}

\section{Acknowledgements}
Oliver Roche-Newton was partially supported by the Austrian Science Fund FWF Projects F5509 and P 30405-N32. Dmitrii Zhelezov was partially supported by the Knuth and Alice Wallenberg Foundation Program for Mathematics 2017. 

We thank Brendan Murphy and Endre Szemer\'edi for helpful conversations.

\end{document}